\documentclass[11pt]{amsart}

\usepackage{amsmath, amssymb, amsthm, comment}
\usepackage{a4wide}

\numberwithin{equation}{section}

\theoremstyle{plain}
\newtheorem{theorem}{Theorem}[section]
\newtheorem{tw}[theorem]{Theorem}
\newtheorem{lem}[theorem]{Lemma}
\newtheorem{propn}[theorem]{Proposition}
\newtheorem{cor}[theorem]{Corollary}

\theoremstyle{definition}
\newtheorem{remark}[theorem]{Remark}

\newtheorem{deft}[theorem]{Definition}

\newcommand\Ind{\mathcal{I}}
\newcommand{\bn}{\mathbb{N}}
\newcommand{\bc}{\mathbb{C}}

\newcommand{\Ad}{\textup{Ad}}

\DeclareMathOperator\Fix{Fix}
\DeclareMathOperator\Aut{Aut}

\newcommand\conv{\star}
\newcommand{\inv}{^{-1}}
\newcommand\conj{\overline}

\newcommand\lone{\mathit{L}^1}
\newcommand\ltwo{\mathit{L}^2}
\newcommand\Linf{\mathit{L}^\infty}

\let\VN=\vn

\newcommand{\ot}{\otimes}
\newcommand{\cop}{\Delta}
\newcommand{\id}{\mathrm{id}}

\newcommand{\Com}{\Delta}

\newcommand{\TX}{\mathsf{X}}
\newcommand{\TY}{\mathsf{Y}}
\newcommand{\TZ}{\mathsf{Z}}

\newcommand{\lX}{\mathsf{R}_{\TX}}

\newcommand{\lY}{\mathsf{R}_{\TY}}
\newcommand{\lZ}{\mathsf{R}_{\TZ}}

\newcommand{\CX}{\mathsf{A}_{\TX}}
\newcommand{\CY}{\mathsf{A}_{\TY}}

\newcommand{\QG}{\mathbb{G}}

\newcommand{\hQG}{\widehat{\QG}}

\newcommand{\wt}{\widetilde}

\newcommand{\wot}{\mathop{\overline\otimes}}

\newcommand{\Hil}{\mathsf{H}}
\newcommand{\Kil}{\mathsf{K}}
\newcommand{\alg}{\mathsf{A}}

\newcommand{\mlg}{\mathsf{M}}

\newcommand{\la}{\langle}
\newcommand{\ra}{\rangle}

\newenvironment{rlist}
{

\begin{enumerate}}
{\end{enumerate}}

\begin{document}

\title[TRO actions and crossed products] {Actions of locally compact
  (quantum) groups on ternary rings of operators, their crossed
  products and generalized Poisson boundaries}


\begin{abstract}
  \noindent
  Actions of locally compact groups and quantum groups on W*-ternary
  rings of operators are discussed and related crossed products
  introduced. The results generalise those for von Neumann algebraic
  actions with proofs  based mostly on passing to the linking von
  Neumann algebra.
They are motivated by the study of fixed point spaces for convolution operators generated by contractive, non-necessarily positive measures, both in the classical and in the quantum context.
\end{abstract}

\author{Pekka Salmi}
\address{Department of Mathematical Sciences,
University of Oulu, PL 3000, FI-90014 Oulun yliopisto, Finland}
\email{pekka.salmi@iki.fi}

\author{Adam Skalski}
\address{Institute of Mathematics of the Polish Academy of Sciences,
ul.~\'Sniadeckich 8, 00--656 Warszawa, Poland }
\email{a.skalski@impan.pl}

\maketitle

The notion of a Poisson boundary, related to spaces of harmonic functions in the context of the probability theory, i.e.\ the fixed points of convolution operators associated to probability measures, has played an important role in the theory of random walks and various aspects of the potential theory for more than 40 years. Starting from the groundbreaking work of Izumi in \cite{Izumi}, the concept (originally introduced for random walks on $\mathbb{Z}$, but later studied for any locally compact group) was extended also to quantum groups. For the history of further developments we refer to the article \cite{KNR2}, where the abstract structure of non-commutative Poisson boundaries is studied in detail and connected to the crossed products of von Neumann algebras.
The quantum extension shows very clearly that from the modern point of view the construction of the Poisson boundary is a special instance of the Choi--Effros product from the theory of operator algebras granting a von Neumann algebra structure to a space of fixed points of a given unital completely positive map on a von Neumann algebra; this is indeed one of the key observations which led to the current work.

When one considers convolution operators, be it in the classical or
the quantum framework, it is natural to analyse not only those
associated to probability measures, but also those arising from general signed
measures. In general the problem of the study of the fixed points
becomes then far more complicated, and even the characterisation of
idempotent signed measures (i.e.\ those whose convolution operators
are idempotent maps) remains unknown for non-abelian groups - we refer
to \cite{ours} for a longer discussion. The questions become more
tractable if one focuses on contractive signed measures. In the
classical and the dual to classical context related issues are studied
for example in the monograph \cite{Chu-Lau}; in the quantum framework
the idempotent problem was solved in \cite{ours}. In both of these
works one saw a natural emergence of the W*-TRO structures
(\emph{W*-ternary rings of operators}). This was a motivation for a
systematic study of the spaces of fixed points of completely
contractive maps which we conduct in this paper.

It turns out that whenever $\mlg$ is a von Neumann algebra and
$P:\mlg \to \mlg$ is a completely contractive normal map, then a Choi--Effros
type construction, exploiting the algebraic properties of $P$
established in \cite{Youngson}, equips the space $\Fix\, P$ with a
unique structure of a W*-TRO, which we may call a \emph{generalized
  noncommutative Poisson boundary} related to $P$. To understand the
structure of the resulting TRO for a contractive convolution operator,
one needs to develop also the notion of (quantum) group actions on
W*-TROs and analyse appropriate crossed products (in the classical
context one can find related  work in \cite{HamanaUnpubl}). Several
theorems follow here relatively easily from their von
Neumann algebraic counterparts, as each W*-TRO is a corner in its
linking von Neumann algebra; some other require certain care, as we
mention below.

Once a satisfactory theory of crossed products is developed, it is
natural to expect a generalisation of the main result of \cite{KNR2},
which would show that the generalised noncommutative Poisson boundary
for an `extended' contractive convolution operator $\Theta_{\mu}$ is
isomorphic to the crossed product of the analogous boundary for the
`standard' convolution operator $R_\mu$ by a natural action of the
underlying group. This is indeed what we prove here, but only for
classical locally compact groups (the positive case considered in
\cite{KNR2} yielded the result for general locally compact quantum
groups). Here we see an example where the passage from the von Neumann
algebra framework to the TRO case is highly non-trivial -- very
roughly speaking the reason is that the Choi--Effros type construction
connects a given concrete initial data (the pair $(\mlg, P)$) with an
\emph{abstract} W*-TRO $\TX$ -- and the linking von Neumann algebra of
$\TX$ arises naturally only once we fix a concrete realisation of the
latter, which need not be related in any explicit way to the original data.

The plan of the article is as follows: in the first section we recall
basic facts regarding the ternary rings of operators, prove a few
technical lemmas  and connect fixed points of the completely
contractive maps to the TROs, introducing the corresponding version of
the Choi--Effros product and discussing its basic properties. In
Section 2 we develop the notion of actions of locally compact groups
on TROs and construct respective crossed products, carefully
developing various points of view on this concept. The next section
extends the construction to the case of locally compact quantum groups
and deals with certain technicalities describing the way in which one
can induce actions on TROs arising as fixed point spaces. These are
applied in Section 4 to the discussion of the TROs arising from fixed
point spaces of contractive convolution operators on locally compact
(quantum) groups.

The angled brackets will denote the closed linear span. Hilbert space
scalar products will be linear on the right. For a locally compact
group $G$ we write $L^2(G)$ for the $L^2$-space with respect to the
\emph{left} invariant Haar measure; the group von Neumann algebra
$\VN(G)$, will be the von Neumann algebra generated by the \emph{left}
regular representation.

\vspace*{0.5 cm}
\noindent
    {\bf Acknowledgement.}\
    The work on this article began during our joint Research in
    Pairs  visit with Matthias Neufang and Nico Spronk to the Mathematisches Forschungsinstitut Oberwolfach in August
2012. We are very grateful to MFO for providing ideal research
conditions and to both Matthias and Nico for many observations which initiated this work. AS would also like to thank Zhong-Jin Ruan for stimulating discussions on the subject of the paper.
AS  was partially supported by the NCN (National Centre of Science) grant
2014/14/E/ST1/00525.

\section{W*-ternary rings of operators and fixed points of
  completely contractive maps}

Recall that a (concrete) \emph{TRO}, i.e.\ a \emph{ternary ring of
  operators}, $\TX$ is a closed subspace of $B(\Hil;\Kil)$, where
$\Hil$ and $\Kil$ are some Hilbert spaces, which is closed under the
ternary product: $(a,b,c) \mapsto ab^*c$. TROs possess natural
operator space structure and in fact can be also characterised
abstractly, as operator spaces with a ternary product satisfying
certain properties, see \cite{NealRusso}. To each TRO $\TX$ one can
associate a C*-algebra $\CX \subset B(\Kil \oplus \Hil)$, called the
\emph{linking algebra} of $\TX$. It is explicitly defined as \[
\CX:= \begin{pmatrix}
         \la \TX \TX^*\ra & \TX \\
                  \TX^* & \la \TX^*\TX \ra
       \end{pmatrix} \subset B(\Kil\oplus \Hil).
\]
We always view $\TX$, $\TX^*$ and the C*-algebras $\la \TX \TX^*\ra$,
$\la \TX^*\TX \ra$ as subspaces of $\CX$.
If $\TX$ and $\TY$ are TROs, then a linear map $\alpha:\TX \to \TY$ is
said to be a \emph{TRO morphism} if it preserves the ternary product.
A TRO morphism  admits a unique extension to a $^*$-homomorphism $\gamma:\CX
\to \CY$ -- this was proved by M.\,Hamana in \cite{Hamana} (see also
\cite{HamanaUnpubl} and \cite{Zettl}) and we will call this map the
\emph{Hamana extension} of $\alpha$. Note that $\gamma$ is defined in
a natural way, so for example if $x\in \TX$ then
$\gamma(xx^*)=\alpha(x) \alpha(x)^*$. Further we call a TRO morphism
$\alpha:\TX \to \TY$ non-degenerate if the linear spans of
$\alpha(\TX)\TY^*\TY$ and $\alpha(\TX)^*\TY \TY^*$ are norm dense
respectively in $\TY$ and $\TY^*$; in other words the space
$\alpha(\TX)$, which is a sub-TRO of $\TY$ by \cite{Hamana}, is a
\emph{non-degenerate sub-TRO} of $\TY$, as defined for example in
\cite{SSTRO}. Then using the Hamana extensions one can easily show
that $\alpha$ is non-degenerate if and only if its Hamana extension
$\gamma:\CX \to \CY$ is non-degenerate (as a $^*$-homomorphism between
C*-algebras) -- see  Proposition 1.1 of \cite{SSTRO} for this result
phrased in terms of sub-TROs.

We say that $\TX$ is a \emph{W*-TRO} if it is weak$^*$-closed in $B(\Hil;\Kil)$. We will usually assume that the TROs we study are non-degenerately represented, i.e.\ $\la \TX\Hil \ra =\Kil$, $\la \TX^* \Kil \ra = \Hil$. The \emph{linking von Neumann algebra} associated to $\TX$, equal to $\CX''$, will be denoted by $\lX$, so that
 \[
\lX:= \begin{pmatrix}
         \la \TX \TX^*\ra'' & \TX \\
                  \TX^* & \la \TX^*\TX \ra ''
       \end{pmatrix} \subset B(\Kil\oplus \Hil).
\]
For a TRO morphism between W*-TROs non-degeneracy will mean that the
linear spans of the spaces introduced in the paragraph above are
weak$^*$-dense in the respective TROs.
The predual of a W*-TRO $\TX$ will be denoted by $\TX_*$; it is not
difficult to see that $\TX_*=\{\omega|_{\TX}: \omega \in (\lX)_*\}$.

There is also an abstract characterisation of TROs and W*-TROs
due to Zettl~\cite{Zettl}, which we now recall.
An \emph{abstract TRO} is a Banach space $\TX$ equipped with
a ternary operation
\[
\{\cdot, \cdot, \cdot\}: \TX\times\TX\times\TX \to \TX
\]
such that the following conditions hold ($x,y,z,u,v \in \TX$):
\begin{enumerate}
\item the operation is linear in the first and the third variable
  and conjugate linear in the second,
\item $\{\{x,y,z\}, u, v\} = \{x, \{u,z,y\}, v\} = \{x, y, \{z, u, v\}\}$;
\item $\|\{x,y,z\}\| \le \|x\| \|y\| \|z\|$;
\item $\|\{x,x,x\}\| = \|x\|^3$.
\end{enumerate}
An \emph{abstract W*-TRO} is an abstract TRO that
is a dual Banach space.
Zettl~\cite{Zettl} proved that these abstractly defined objects
have concrete representations as TROs and W*-TROs, respectively.

The next result is a W*-version of the fact due to Hamana regarding
images of TROs, observed in \cite{BLM}.

\begin{lem}
If $\TX$ and $\TY$ are W*-TROs, and $\alpha:\TX \to \TY$ is a normal
TRO morphism, then $\alpha(\TX)$ is a W*-TRO.
\end{lem}
\begin{proof}
This is proved in Section 8.5.18 in \cite{BLM}. The main idea is as follows: the kernel of $\alpha$ is a sub-TRO of $\TX$. Thus it can be written as $f\TX$, where $f$ is a central projection in the von Neumann algebra $\la \TX \TX^* \ra ''$. The morphism $\alpha':(1-f)\TX \to \TY$ given by the restriction of $\alpha$ is then easily seen to be injective; moreover $\alpha'(\TX)=\alpha(\TX)$. Then we deduce that the unit ball of $\alpha'(\TX)$ is the image of the unit ball of $\TX$, hence it is weak$^*$-compact (so in particular weak $^*$-closed).
\end{proof}

Proposition 3.1 of \cite{SSTRO} shows that a TRO morphism $\alpha:\TX \to
\TY$ is non-degenerate if and only if the W*-TRO $\alpha(\TX)$ is
non-degenerately represented.
The last lemma can be used to note that Hamana extensions can be
considered also in the W*-category and moreover have the expected properties.

\begin{propn} \label{extend}
Let $\TX$ and $\TY$ be W*-TROs and let $\alpha:\TX \to \TY$ be a
normal TRO morphism. Then there exists a unique normal
$^*$-homomorphism $\beta:\lX\to \lY$ such that
\[ \beta \begin{pmatrix} 0 & x \\ 0 &0 \end{pmatrix} = \begin{pmatrix} 0 & \alpha(x) \\ 0 &0 \end{pmatrix}, \qquad x\in \TX. \]
Moreover $\beta|_{\CX}$ is the Hamana extension discussed above, the extension construction preserves the composition, and moreover
\begin{rlist}
\item $\alpha$ is  injective if and only if $\beta$ is injective;
 \item $\alpha$ is non-degenerate if and only if $\beta$ is unital.
 \end{rlist}
\end{propn}

\begin{proof}
There are at least two ways to see the first statement (the rest is
relatively easy). In the first step one observes that we can assume
that $\alpha$ is surjective, using the last lemma (indeed, if
$\TZ=\alpha(\TX)$, then $\lZ$ is a von Neumann subalgebra of
$\lY$).

Now we can either proceed directly, as in \cite{SSTRO}, using
non-degeneracy, or first pass to the situation where $\alpha$ is
isometric, quotienting out its kernel (this leads to another W*-TRO,
as follows from \cite{Hamana}) and then use the proof of Corollary
3.4 in \cite{Solel}. The reason we cannot use  this corollary directly
is that we need to verify that the map obtained there (or in fact
rather via Theorem 2.1 of that paper) coincides with the
weak$^*$-continuous extension of the Hamana extension. This however
can be checked directly, following the arguments in Corollary 3.4 and
Lemma 2.5 of \cite{Solel}.

Injectivity of $\alpha$ (respectively, $\beta$) is equivalent to $\alpha$ (respectively, $\beta$) being isometric; thus \cite{Solel} implies that injectivity of $\alpha$ is equivalent to that of $\beta$.

For a W*-TRO $\TX\subset B(\Hil;\Kil)$ it is elementary to check that
$\TX$ is non-degenerately represented if and only if $\lX$ contains
the unit of $B(\Kil \oplus \Hil)$. This together with the comments
before the proposition implies the last statement.
\end{proof}

Note that in the situation above by boundedness and normality of the maps in question,  we have the following consequence of the algebraic form of Hamana extensions (in which we view both $\TX$ and $\la \TX \TX^* \ra ''$ as  subspaces of $\lX$):
\begin{equation} \beta (z) \alpha(x) = \alpha(zx), \qquad x \in \TX,
  z\in \la \TX \TX^* \ra ''. \label{betaalpha}\end{equation}
If $\TX$ happens to be a von Neumann algebra, $\lX \cong M_2(\TX)$; if $\TY$ is another von Neumann algebra and we assume that $\alpha:\TX\to\TY$ is a $^*$-homomorphism, then $\beta$ is the usual matrix lifting of $\alpha$. Finally we note an easy observation which will be useful later.

\begin{cor} \label{extendconverse}
Let $\TX$ and $\TY$ be W*-TROs and let $\beta:\lX \to \lY$ be a normal
$^*$-homomorphism. Then $\beta$ is the Hamana extension of a normal
TRO morphism between $\TX$ and $\TY$ if and only if $\beta(\TX)\subset
\TY$. If $\TX$ and $\TY$ are respectively non-degenerately represented
in $B(\Hil_1;\Kil_1)$ and in $B(\Hil_2;\Kil_2)$, then the conditions
above are equivalent to the equality
\[ P_{\Kil_2}\beta(P_{\Kil_1} x P_{\Hil_1})P_{\Hil_2} = \beta(P_{\Kil_1} x P_{\Hil_1})\]
being valid for all $x \in \lX$.
\end{cor}

\begin{proof}
A simple calculation shows that if $\beta(\TX) \subset \TY$, then
$\beta|_{\TX}$ is a TRO morphism. Then the equivalence follows from
the uniqueness of Hamana extensions and the second statement is an
easy consequence of the definitions of $\lX$ and $\lY$.
\end{proof}

Given two W*-TROs $\TX\subset B(\Hil_1;\Kil_1)$ and
$\TY\subset B(\Hil_2;\Kil_2)$, we can naturally consider the W*-TRO
$\TX \wot\TY$ defined as the weak$^*$ closure of the algebraic tensor
product $\TX \odot \TY \subset B(\Hil_1 \ot \Hil_2; \Kil_1\ot\Kil_2)$.
The fact that it is closed under the ternary product can be easily checked.
Note that if $\mlg$ is a von Neumann algebra, then  we have a natural
identification of $\mathsf{R}_{\TX \wot \mlg}$ with $\lX \wot \mlg$;
this will be of use later.

Similarly note for the future use that if $\mathsf{Z}\subset B(\Hil)$
is a weak$^*$-closed subalgebra and $P\in \mathsf{Z}$ is a projection,
then $P\mathsf{Z}$ is weak$^*$-closed. This implies that if say
$Q\in \mathsf{Z}$ is another projection and  $\mathsf{W}\subset B(\Kil)$ is
a weak$^*$-closed subalgebra, then
\[  P \mathsf{Z}Q \wot \mathsf{W} = (P \ot I_{\Kil})(\mathsf{Z} \wot
\mathsf{W}) (Q \ot I_{\Kil}).  \]
As usual, $\mathsf{Z} \wot \mathsf{W}$ denotes the
 weak$^*$ closure  of algebraic tensor product
$\mathsf{Z} \odot \mathsf{W}$ inside $B(\Hil) \wot B(\Kil)$.

Finally recall (for example from Chapter 7 of \cite{EfR}) that if $\TX$ and $\TY$ are dual operator spaces, then
their \emph{Fubini tensor product} $\TX \ot_F \TY$ is defined
abstractly as the operator space dual of
$\TX_* \widehat{\otimes} \TY_*$; if $\TX$ and $\TY$ are
weak$^*$-closed subspaces of say $B(\Hil)$ and $B(\Kil)$,
then $\TX \ot_F \TY$ can be realised as
\[
\{u\in B(\Hil) \wot B(\Kil) : (\omega\ot \id)u\in \TY
\text{ and } (\id \ot\sigma)u\in \TX\text{ for every }
\omega\in B(\Hil)_*, \sigma\in B(\Kil)_*\},
\]
Clearly, $\TX \wot \TY$ is contained in $\TX \ot_F \TY$.

\begin{lem} \label{lemma:fubini}
Let $\TX$ and $\TY$ be dual operator spaces that are weak$^*$
completely contractively complemented in von Neumann
algebras \textup{(}note that in particular
W*-TROs satisfy these assumptions\textup{)}. Then the natural
weak$^*$-continuous completely isometric embedding $\TX\wot \TY
\hookrightarrow \TX \ot_F \TY$ is in fact an isomorphism.
\end{lem}
\begin{proof}
By the assumptions there are von Neumann algebras $R_\TX$ and
$R_\TY$ containing $\TX$ and $\TY$, respectively,
and normal completely contractive projections
$P_\TX: R_\TX\to \TX$  and $P_\TY: R_\TY\to \TY$.
The algebraic tensor product $P_\TX\odot P_\TY$ extends uniquely to a
normal map $P_\TX\ot P_\TY$
from $R_\TX\wot R_\TY = R_\TX\ot_F R_\TY $ to $\TX \ot_F \TY$
(see Chapter 7 of \cite{EfR} and Proposition 4.3 of \cite{EfR-kac}).
As $\TX \ot_F \TY\subset R_\TX\wot R_\TY$, the uniqueness of extensions
implies that $P_\TX\ot P_\TY$ is the identity map when restricted to
$\TX \ot_F \TY$. Let $u\in \TX \ot_F \TY$ and let
$(u_i)_{i \in \Ind}$ be a net in the algebraic tensor product
$R_\TX\odot R_\TY$ that converges to $u$ in the weak* topology.
Then
\[
u = (P_\TX\ot P_\TY)(u) =
\mathrm{w\sp *-}\lim_{i \in \Ind} (P_\TX\odot P_\TY)(u_i)\in \TX\wot \TY.
\]
\end{proof}

Consider then a  TRO morphism $\alpha:\TX \to \TY$. It follows from
\cite{Hamana} that $\alpha$ is completely contractive. Moreover,
Proposition 1.1 of \cite{HamanaUnpubl}
implies that if $\TZ$ is another W*-TRO, then the map
$\id_{\TZ} \ot \alpha$ extends uniquely to a completely contraction
from $\TZ \ot_F \TX$ to $\TZ \ot_F \TY$ -- this does not require
that the original map is normal. If $\alpha$ is in addition
normal, the resulting extension is also normal, as follows for example from
the identification of the predual of the Fubini tensor product as the
projective tensor product of the preduals of the individual factors.
So when $\alpha$ is normal, we can view $\id_{\TZ}\ot\alpha$
as a normal TRO morphism from $\TZ  \wot \TX$ to $\TZ \wot \TY$.
If we want to stress that we are working with a not necessarily normal
extension we will write $\id_{\TZ} \ot_F \alpha$.

The celebrated Choi--Effros construction equips a fixed point space of a completely positive map with a C*-algebra structure. Below we present an analogous result for completely contractive maps and TRO structures.

The first proposition is essentially a theorem of Youngson
\cite{Youngson} (see also Theorem 4.4.9 in \cite{BLM}).

\begin{propn} \label{TROconst}
Let $\alg$ be a C*-algebra and let $P:\alg \to \alg$ be a completely contractive projection. Then $\wt{\TX}:=P(\alg)$ possesses a TRO structure, with the product given by the formula
\[ \{a,b,c \}:=P(ab^*c), \qquad a,b,c \in \wt{\TX}.\]
Denote the resulting TRO by $\TX$. Then the identity map $\iota:
\wt{\TX} \to \TX$ is a completely isometric isomorphism
\textup{(}where
$\wt{\TX}$ inherits the operator space structure from $\alg$
and $\TX$ is an operator space as a TRO\textup{)}. If
$\alg$ is a von Neumann algebra and $\wt{\TX}$ happens to be
weak$^*$-closed, then  $\TX$ is a W*-TRO and
$\iota:  \wt{\TX} \to \TX$ is in addition a homeomorphism for
weak$^*$ topologies.
\end{propn}

\begin{proof}
The fact that the displayed formula defines a TRO structure (with
the norm induced from $\alg$) is the Theorem of \cite{Youngson}
(p.\ 508) -- it follows also from the abstract description due to
Zettl mentioned earlier. The map $\iota$ is thus an isometry.
Applying the same construction to $P^{(n)}:M_n(\alg)\to M_n(\alg)$
gives a TRO based on $M_n(\wt\TX)$ and this TRO is naturally
isomorphic to $M_n(\TX)$.
Proposition 2.1 of \cite{Hamana} implies that this isomorphism is an
isometry. Thus $\iota$ is in fact a complete isometry.
The second part follows from the uniqueness of a predual of a W*-TRO
(Proposition 2.4 of \cite{effros-ozawa-ruan}).
\end{proof}

\begin{theorem} \label{CE_TRO}
Let $\mlg$ be a von Neumann algebra and let $P:\mlg \to \mlg$ be a completely contractive normal map. Consider the space $\Fix P = \{x\in \mlg: Px=x\}$. Then $\Fix P$ is a weak$^*$-closed subspace of $\mlg$, so in particular a dual operator space. It  possesses a unique ternary product which makes it a W*-TRO. It is explicitly given by the
formula
\[  \{a,b,c \}:= \wt{P}_{\beta} (ab^*c), \qquad a,b,c \in \Fix P,\]
where $\beta$ is a fixed free ultrafilter,
\[ \wt{P}_{\beta} (x) = \beta- \lim_{n\in \bn} \frac{1}{n} \sum_{k=0}^{n-1} P^k (x), \qquad x \in \mlg, \]
and the last limit is understood in a weak$^*$ topology.
\end{theorem}

\begin{proof}
To show the existence of the ternary product described above it
suffices to verify that $P_{\beta}:\mlg \to \mlg$ is a completely
contractive projection onto $\Fix P$. The fact that $P_{\beta}$ is a
projection onto $\Fix P$ follows by standard Ces\`aro limit arguments;
the (complete) contractivity of $P_{\beta}$ follows from the analogous
property of $P$ and the fact that a weak$^*$ limit of contractions is
a contraction. The uniqueness of the ternary product follows once
again from Proposition 2.1 of \cite{Hamana}.
\end{proof}

\begin{remark}
Let us stress that Proposition \ref{TROconst} implies in particular
that the W*-TRO structure of  $\Fix P$ does not depend on the
choice of the ultrafilter in the above proof (although the map
$\wt{P}_{\beta}$ may well do).
\end{remark}

The following result is an abstract extension of Proposition 3.3.1 of \cite{Chu-Lau}, the proof is essentially the same.

\begin{propn} \label{uniqueifnormal}
Suppose that the assumptions of Theorem \ref{CE_TRO} hold.
If there exists a normal
\textup{(}i.e.\ weak$^*$--weak$^*$-continuous\textup{)}
projection $Q:\mlg \to \Fix P$ such that $Q \circ P = P\circ Q$,
then $\wt{P}_{\beta} = Q$ for any free ultrafilter $\beta$  .
\end{propn}

\begin{proof}
Recall that a normal projection is necessarily bounded. Thus we have for each $x \in \mlg$ (and a free ultrafilter $\beta$)
\begin{align*} \wt{P}_{\beta}(x) &= Q \wt{P}_{\beta}(x) =  Q (\beta{-} \lim_{n\in \bn} \frac{1}{n} \sum_{k=0}^{n-1} P^k (x)) \\&=  \beta{-} \lim_{n\in \bn} Q(\frac{1}{n} \sum_{k=0}^{n-1} P^k (x)) =
\beta{-}\lim_{n\in \bn} \frac{1}{n} \sum_{k=0}^{n-1} P^k (Qx) = Qx.\end{align*}
\end{proof}

\section{Actions of locally compact groups on W*-TROs and resulting
  crossed products}

In this section we discuss actions of locally compact groups on
W$^*$-TROs and the associated crossed products.
Analogous study in the operator space context was undertaken
in \cite{HamanaUnpubl}; we will comment on some specific analogies at
the end of this section.

\begin{deft} \label{def:group-act}
Let $G$ be a locally compact group, and let $\TX$ be a W*-TRO.
Denote by $\Aut(\TX)$ the set of all normal automorphisms of $\TX$,
i.e.\ normal bijective TRO morphisms from $\TX$ onto itself.
A (continuous) \emph{action} of $G$ on $\TX$ is a homomorphism
$\alpha:G \to \Aut(\TX)$ such that for each $x\in \TX$
the map $\alpha^x: G \to \TX$ defined by
\[ \alpha^x(s) = (\alpha (s)) (x),\qquad s \in G,\]
is weak$^*$-continuous.
We shall write $\alpha_s = \alpha(s)$ for $s\in G$.
\end{deft}

The continuity condition above has several equivalent formulations
which can be deduced from Sections 13.4 and 13.5 of
\cite{Stratila}. We record one of them in the following proposition.

\begin{propn} \label{propn:action-cts}
Let $G$ be a locally compact group, and let $\TX$ be a W*-TRO. A
homomorphism $\alpha:G \to \Aut(\TX)$  is a continuous action of $G$ on
$\TX$  if and only if the map
$G \times \TX_* \ni (s,\varphi) \mapsto \varphi \circ \alpha_s \in\TX_*$
is norm-continuous.
\end{propn}

We are ready to connect the action of $G$ on $\TX$ with the action on $\lX$.

\begin{tw}\label{groupaction}
Let $\alpha$ be an action of a locally compact group $G$ on a W*-TRO $\TX$. Then it possesses a unique extension to an action of $G$ on $\lX$.
\end{tw}
\begin{proof}
First fix $g \in G$ and extend $\alpha_g \in \Aut (\TX)$ to  a normal automorphism $\beta_g \in \Aut(\lX)$ via Proposition \ref{extend}.

The uniqueness of the extensions implies that the resulting family
$\{\beta_g:g \in G\}$ defines a homomorphism $\beta:G \to
\Aut(\lX)$. It remains to check that it satisfies the continuity
requirement. We do it separately for each corner of the map $\beta$,
presenting the argument only for the upper left corner.

Take $z\in \la \TX \TX^* \ra ''$  and consider the map
$\beta^z:G \to \lX$. We need to show it is weak$^*$-continuous. As all
the maps in question are contractive and we may assume that
$\TX$ is non-degenerately represented in $B(\Hil;\Kil)$, it suffices
to check that for all $\xi \in \Kil$, $x \in \TX$, $\eta \in \Hil$,
and a net of elements $(s_i)_{i\in \Ind}$
of $G$ converging to $e\in G$, we have
\[ \la \xi, \beta_{s_i}(z) x \eta \rangle \stackrel{i\in \Ind}{\longrightarrow} \la \xi, zx \eta \ra.\]
Note that, by \eqref{betaalpha},
\begin{align*}
\la \xi, \beta_{s_i}(z) x \eta \rangle  = \la \xi, \beta_{s_i}(z) \alpha_{s_i} (\alpha_{s_i^{-1}} (x)) \eta \rangle =
\la \xi,  \alpha_{s_i} (z\alpha_{s_i^{-1}} (x)) \eta \rangle = \omega_{\xi,\eta} \circ \alpha_{s_i} \left( z\alpha_{s_i^{-1}} (x) \right),
\end{align*}
so putting $\omega:=\omega_{\xi, \eta}$ we obtain
\begin{align*}
\la \xi, \beta_{s_i}(z) x \eta \rangle  - \la \xi, zx \eta \ra =
(\omega \circ \alpha_{s_i} - \omega)(z\alpha_{s_i^{-1}} (x)) +
\omega   \left(z\alpha_{s_i^{-1}} (x) -zx \right).
\end{align*}
Applying Proposition~\ref{propn:action-cts}, we see that the upper
left corner of $\beta^z$ is weak*-continuous.
The remaining parts of the proof follow analogously.
\end{proof}

For an action $\alpha$ of $G$ on a W*-TRO $\TX$ we define the fixed point space $\Fix {\alpha}$ as
\[ \Fix{\alpha}=\{x\in \TX: \forall_{g \in G}\, \alpha_g(x) = x\}.\]
It is clear that $\Fix{\alpha}$ is a W*-sub-TRO of $\TX$.

\begin{cor} \label{fixedcorner}
Assume that $\TX$ is a W*-TRO that is non-degenerately represented
in some $B(\Hil;\Kil)$, $\alpha$ is an action of $G$ on $\TX$ and
$\beta$ is an action of $G$ on $\lX$ introduced in Theorem
\ref{groupaction}. Then $\Fix{\alpha} = P_{\Kil} (\Fix{\beta})
P_{\Hil}$.
\end{cor}

\begin{proof}
Let $x\in \TX$, $g \in G$. If $\alpha_g(x) =x$ then we also have
$\beta_g(x) =x$, and naturally $P_{\Kil} x  P_{\Hil}=x$, which proves
the inclusion `$\subset$' in the desired equality. On the other hand
if $x=P_{\Kil} z P_{\Hil}$ for some $z \in \Fix{\beta}$ then
\begin{align*}
\alpha_g(x) = \beta_g(x) = \beta_g (P_{\Kil} z P_{\Hil}) = \beta_g (P_{\Kil}) \beta_g (z) \beta_g(P_{\Hil}) =  P_{\Kil} z P_{\Hil} =x,
\end{align*}
where we used the fact that $\beta_g$ is a homomorphism and that (by construction) it preserves the projections $P_{\Kil}$ and $P_{\Hil}$.
\end{proof}

We now discuss the connection of pointwise actions defined above
with their integrated incarnations. The interplay between the two
plays a crucial role in \cite{HamanaUnpubl} -- the situation studied
there is however subtler, as the W*-context (as opposed to the
C*-problems studied by Hamana) and presence of linking von Neumann
algebras leads to certain simplifications. Recall that if $G$ is a
locally compact group, then $L^{\infty}(G)$ admits a natural coproduct
(see also Section \ref{LCQG}) $\Com:L^{\infty}(G) \to L^{\infty}(G)
\wot L^{\infty}(G)$, defined via the isomorphism $L^{\infty}(G) \wot
L^{\infty}(G) \cong L^{\infty}(G\times G)$ and the formula
\[ \Com(f)(g,h) = f(gh), \qquad f \in L^{\infty}(G), g, h \in G.\]

\begin{theorem}\label{avatars}
Suppose that $\alpha:G \to \Aut(\TX)$ is an action of a locally
compact  group $G$ on a W*-TRO $\TX$. Then there exists a unique
map $\pi_{\alpha}:\TX\to L^{\infty}(G) \wot \TX$ such that for each
$f\in L^1(G)$, $\phi \in \TX_*$ and $x \in \TX$ we have
\begin{equation}
(f \ot \phi )(\pi_{\alpha}(x)) = \int_G f(g) \phi(\alpha_{g^{-1}}(x))\, dg.
\label{point>int}
\end{equation}
Moreover, if we write $\gamma:=\pi_{\alpha}$, then $\gamma$ is an
injective, normal, non-degenerate TRO morphism such that
\begin{equation}
(\Com  \ot \id_{\TX})\circ \gamma = ( \id_{L^{\infty}(G)} \ot
  \id_{\TX})\circ \gamma. \label{actionequation}
\end{equation}
Conversely, if $\gamma:\TX \to  L^{\infty}(G)  \wot \TX$
is an injective, normal, non-degenerate TRO morphism
satisfying \eqref{actionequation},
then there exists a unique action $\alpha$ of $G$ on $\TX$ such that
$\gamma = \pi_{\alpha}$.
\end{theorem}

\begin{proof}
We may assume that $\TX$ is non-degenerately represented in $B(\Hil;\Kil)$.

Assume first that we are given an action $\alpha:G \to \Aut(\TX)$ and
extend it pointwise, using Theorem \ref{groupaction}, to a continuous
action $\beta:G \to \Aut(\lX)$. Discussion in Section 18.6 of
\cite{Stratila} implies that there exists an injective, normal, unital
$^*$-homomorphism $\pi_{\beta}:\lX \to  L^{\infty}(G) \wot \lX$ such
that
for each $f\in L^1(G)$, $\phi \in (\lX)_*$ and $z \in \lX$ we have
\begin{equation}
(f \ot \phi)(\pi_{\beta}(z)) = \int_G f(g) \phi(\beta_{g^{-1}}(z))\,dg
\label{phifM}\end{equation}
and $( \Com \ot \id)\circ \pi_{\beta} = (\id \ot \pi_{\beta})\circ \pi_{\beta}$
(note that our formulas are formally different from Stratila's: the
difference is however only in the fact that we choose to work with
maps taking values in $L^{\infty}(G) \wot \lX$, and not in $  \lX \wot
L^{\infty}(G)$, which allows us to work with the standard coproduct of
$L^{\infty}(G)$, and not with the opposite one as it is in
\cite{Stratila}).
Consider the map $\pi_{\beta}|_{\TX}$. We want to show that it takes
values in $\TY:=L^{\infty}(G)\wot \TX$. Considerations before
Lemma \ref{lemma:fubini} imply that the latter space is a W*-TRO
equal to $( I_{L^{2}(G)} \ot P_{\Kil}) (L^{\infty}(G)\wot  \lX ) (
I_{L^{2}(G)} \ot P_{\Hil})$; moreover $\lY=  L^{\infty}(G) \wot \lX
$. Let then $x \in \TX$. It suffices to show that for any $f\in
L^1(G)$, $\phi \in (\lX)_*$ we have
\[ (f \ot \phi) (\pi_{\beta}(x)) = (f \ot \phi) \bigl(( I_{L^{2}(G)} \ot P_{\Kil}) \pi_{\beta}(x) ( I_{L^{2}(G)} \ot P_{\Hil})\bigr).\]
This however follows immediately from \eqref{phifM} once we note that
$\beta_g(x) = \alpha_g(x)$ for all $g \in G$ and that $y\mapsto
\phi(P_{\Kil} y P_{\Hil})$ is a normal functional on $\lX$.
Thus we showed that $\gamma:=\pi_{\beta}|_{\TX}$ maps $\TX$ into
$\TY$. It is an injective, normal TRO morphism satisfying
\eqref{actionequation} (as a restriction of an injective, normal
$^*$-homomorphism satisfying \eqref{actionequation}). By the
uniqueness of Hamana extensions and the above identification of $\lY$
we deduce that $\pi_{\beta}$ is the Hamana extension of $\gamma$, so
that non-degeneracy of $\gamma$ follows from unitality of
$\pi_{\beta}$ via Proposition \ref{extend}.

Assume now that  $\gamma:\TX \to  L^{\infty}(G) \wot \TX $
is an injective, normal, non-degenerate TRO morphism
satisfying  the action equation \eqref{actionequation}.
Again write $\TY= L^{\infty}(G) \wot \TX $ and let $\pi:\lX\to \lY=
L^{\infty}(G) \wot \lX $ denote the Hamana extension of
$\gamma$. Proposition \ref{extend} implies that $\pi$ is a unital,
injective normal $^*$-homomorphism.  Normality of $\pi$ and $\Com$
implies that it suffices to check the validity of the action equation with $\gamma$ replaced by $\pi$ on a weak$^*$-dense subset; this follows in
turn from the computations of the type ($x,z \in \TX$):
\begin{align*} (\Com \ot \id_{\TX} ) (\pi(xz^*)) &= (\Com \ot \id_{\TX} ) (\gamma(x) \gamma(z)^*) = (\Com \ot \id_{\TX} ) (\gamma(x)) (\Com \ot \id_{\TX} ) (\gamma(z))^* \\&=
 ( \id_{L^{\infty}(G)} \ot \gamma) (\gamma(x)) (\id_{L^{\infty}(G)} \ot \gamma)(\gamma(z))^* = (\id_{L^{\infty}(G)} \ot \pi)(\gamma(x) \gamma(z)^*) \\&=
 (\id_{L^{\infty}(G)} \ot \pi)(\pi(xz^*)). \end{align*}
The Proposition in Section 18.6 of \cite{Stratila} (or rather its
left version) implies that there exists an action $\beta:G \to
\Aut(\lX)$ such that $\pi=\pi_{\beta}$, where $\pi_{\beta}$ is defined
via formula \eqref{phifM}. It remains to show that for each $g \in G$
the map $\alpha_g:=\beta_g|_{\TX}$ takes values in $\TX$, as then it will be
easy to check that the family
$(\alpha_g)_{g\in G}$ defines an action of $G$ on
$\TX$ and that $\gamma$ arises from this action via the formulas given
in the theorem. Fix then $x\in \TX$ and define $z_g=\beta_{g^{-1}}(x)$
for each $g \in G$. Then $z:G \to \lX$ is a weak$^*$-continuous
function and we know that for all $f\in L^1(G)$ and all $\phi \in
(\lX)_*$ such that $\phi(\TX)=\{0\}$ we have
\[ \int_G f(g) \phi(z_g)\, dg =0.\]
But then we deduce immediately that $\phi(z_g)$ is $0$ for almost all $g \in G$, and as it is a continuous function it must actually be $0$ everywhere. This in turn means that $z_g \in \TX$ for all $g \in G$, which ends the proof.
\end{proof}

The next proposition is also familiar from the von Neumann algebraic context.

\begin{propn} \label{pialpha}
Suppose that $\alpha:G \to \Aut(\TX)$ is an action of a locally
compact  group $G$ on a W*-TRO $\TX$ and assume that $\TX$ is
non-degenerately represented in $B(\Hil;\Kil)$. Then the map
$\pi_{\alpha}$ introduced in Theorem \ref{avatars} may be viewed as a
faithful representation of  $\TX$ in $B(L^2(G; \Hil); L^2(G;\Kil))$,
and we have for all $x \in \TX$ and $\zeta \in L^2(G; \Hil)$
\[
(\pi_{\alpha}(x) (\zeta)) (g) = \alpha_{g^{-1}}(x) \zeta(g)
\qquad\text{for almost every } g \in G. \]
\end{propn}
\begin{proof}
The fact that $\pi_{\alpha}$ can be viewed as a faithful representation
of $\TX$ in $B(L^2(G; \Hil); L^2(G;\Kil))$ follows from Theorem
\ref{avatars}.

It remains to prove the displayed formula.
We identify $L^2(G; \Hil)$ with $L^2(G) \ot \Hil$
and $L^2(G; \Kil)$ with $ L^2(G) \ot \Kil$, and let $\xi\in \Hil$,
$\eta\in \Kil$ and $f,h\in\ltwo(G)$.
Then
\[
\la h \ot \eta, \pi_{\alpha}(x)(f \ot \xi) \ra
= \int \la  h(g) \eta, \bigl(\pi_{\alpha}(x)(f \ot \xi)\bigr)(g)  \ra \, dg.
\]
By Theorem~\ref{avatars}, the left-hand side of the previous identity
is equal to
\begin{align*}
\int f(g)\conj{h(g)}\la \eta , \alpha_{g\inv}(x)\xi \ra\, dg
= \int  \la h(g) \eta , f(g)\alpha_{g\inv}(x)\xi  \ra\, dg.
\end{align*}
Then the displayed formula follows by density.
\end{proof}

In the next lemma we show how implemented actions of $G$ on W*-TROs look like.

\begin{lem}
Assume that $\TX$ is a concrete W*-TRO in $B(\Hil;\Kil)$, that $\sigma:G \to B(\Hil)$, $\tau:G \to B(\Kil)$ are so-continuous representations of $G$ and
that for each $g\in G$ and $x \in \TX$ the operator $\tau_g x \sigma_{g}^*$ belongs to $\TX$. Then the map $\alpha: G \to \Aut(\TX)$ defined by
\[\alpha_g(x) = \tau_g x \sigma_{g}^* , \qquad g \in G, x \in \TX,\]
is an action of $G$ on $\TX$.
\end{lem}
\begin{proof}
Straightforward checks: we first observe that $\alpha_g$ is indeed a
normal TRO automorphism of $\TX$ and then verify $\alpha$ is a
homomorphism and that the continuity conditions are satisfied.
\end{proof}

In fact all actions of groups on TROs can be put in this form, at the
cost of extending of the TRO in question. This is analogous to the
crossed product construction for the actions of groups on von Neumann
subalgebras.

\begin{lem}\label{equivcrprod}
Let $\alpha: G \to \Aut(\TX)$  be an action of a locally compact group $G$ on a W*-TRO $\TX$. Assume that $\TX$ is concretely represented as a W*-sub-TRO of $B(\Hil;\Kil)$. Let $\pi:=\pi_{\alpha}:\TX\to B(L^2(G) \ot \Hil; L^2(G) \ot \Kil)$ be the representation of $\TX$ introduced Proposition \ref{pialpha}
and let $\tau=\lambda \ot I_{\Kil}$, $\sigma = \lambda \ot I_{\Hil}$ denote the respective amplifications of the left regular representation of $G$.
Then the space
\begin{equation}
 \mathrm{w\sp *-}\textup{cl }  \textup{Lin} \{(\VN(G) \ot  I_{\Kil})\pi(\TX)  \} \label{crossedform} \end{equation}
is equal to
\begin{align*}
\mathrm{w\sp *-}\textup{cl } \textup{Lin} \{\tau_g \pi(x) : g \in G, x \in
\TX\}
&=\mathrm{w\sp *-}\textup{cl }  \textup{Lin} \{\tau_g \pi(x) \sigma_{g'}: g,g' \in G\}\\
&=\mathrm{w\sp *-}\textup{cl } \textup{Lin} \{ \pi(x) \sigma_{g}: g \in
G, x \in \TX\}
\end{align*}
and is a W*-TRO. Moreover, we have the following formula:
\begin{equation} \pi(\alpha_g(x)) =
\tau_g \pi(x) \sigma_{g}^*,\qquad g \in G, x \in \TX.
\label{implementedalpha} \end{equation}
\end{lem}
\begin{proof}
It suffices to prove the formula \eqref{implementedalpha}, the rest is
based on easy checks.

For $x\in \TX$, $\xi\in L^2(G; \Hil)$ and a.e.\ $g,h\in G$, we have
\begin{align*}
(\tau_g \pi(x) \sigma_g^* \xi)(h) &= (\pi(x)\sigma_{g\inv} \xi)(g\inv h)
 = \alpha_{h\inv g}(x)\bigl((\sigma_{g\inv} \xi)(g\inv h)\bigr)
 \\&= \alpha_{h\inv}\bigl(\alpha_{g}(x)\bigr) \bigl(\xi(h)\bigr)
 = \pi\bigl(\alpha_{g}(x)\xi\bigr)(h),
\end{align*}
as claimed.
\end{proof}

\begin{deft} \label{defcrprod}
Let $\alpha: G \to \Aut(\TX)$  be an action of a locally compact group
$G$ on a W*-TRO $\TX$. The W*-TRO described by  formula
\eqref{crossedform} above is called the \emph{crossed product} of $\TX$ by
$\alpha$ and is denoted $G \ltimes_{\alpha} \TX$.
\end{deft}

\begin{propn} \label{crossedcorner}
Let $\alpha: G \to \Aut(\TX)$  be an action of a locally compact group
$G$ on a W*-TRO $\TX$ and let $\beta:G \to \Aut(\lX)$ be an action of
$G$ on $\lX$ introduced in Theorem \ref{groupaction}. Then the crossed
product $G \ltimes_{\alpha} \TX$ is the corner in the crossed product
$G \ltimes_{\beta} \lX$: if we start from $\TX$ represented
non-degenerately in $B(\Hil;\Kil)$, we obtain
\[ (I_{L^2(G)} \ot P_{\Kil})(G \ltimes_{\beta} \lX)  (I_{L^2(G)} \ot P_{\Hil}) = G \ltimes_{\alpha} \TX.\]
\end{propn}
\begin{proof}
An immediate consequence of the fact that  the space
\[(I_{L^2(G)} \ot P_{\Kil})
\bigl((\VN(G) \ot  I_{\Kil \oplus \Hil})\pi_{\beta}(\lX)\bigr)
 (I_{L^2(G)} \ot P_{\Hil})\]
coincides with
$(\VN(G) \ot  I_{\Kil})\pi_{\alpha}(\TX)$, established in the proof
of Theorem \ref{avatars} and the weak$^*$ density of respective
spaces.
\end{proof}

The following corollary is now easy to observe, once again using the von
Neumann algebra result and referring to the properties of Hamana extensions.

\begin{cor}
The crossed product $G \ltimes_{\alpha} \TX$ does not depend on the choice of the original faithful non-degenerate representation of $\TX$.
\end{cor}

The final result in this section explains the connection between the
definition of the crossed product introduced above and that of Hamana
in \cite{HamanaUnpubl}. Before we formulate it we need to introduce
another action: if $\alpha:G\to \Aut(\TX)$ is an action, then
$\Ad_{\rho} \ot \alpha $ is an action of $G$ on the W*-TRO
$B(L^2(G))\wot \TX$ given by the formula
\begin{equation}
(\Ad_{\rho} \ot \alpha)_g (z) = (\Ad_{\rho_g} \ot \alpha_g )(z),
  \qquad z \in  B(L^2(G)) \wot \TX, \label{tensoraction}
\end{equation}
where $\rho:G \to B(L^2(G))$ is the right regular representation and
we take the convention that $\Ad_{\rho_g}(z) = \rho_g z\rho_g^*$ for
$z \in  B(L^2(G))$. Note for further use the following fact: if we
write $\delta:=\Ad_{\rho} \ot \alpha$, then the corresponding map
$\pi_{\delta}:  B(L^2(G)) \wot \TX \to L^{\infty}(G) \wot B(L^2(G))
\wot \TX$ is given explicitly by the formula:
\begin{equation} \label{Adrhoalpha}
\pi_{\delta} (z) = \chi_{12} (V_{12}^* (\id_{B(L^2(G))} \ot \pi_{\alpha})(z)V_{12}), \qquad z \in B(L^2(G)) \wot \TX ,
\end{equation}
where $V\in B(L^2(G) \ot L^2(G))$ is the right multiplicative unitary
of $G$ (see Subsection \ref{LCQG}) and $\chi_{12}$ flips the first two
legs of the tensor product.

\begin{propn} \label{crossedfixed}
If $\alpha:G \to \Aut(\TX)$ is an action of $G$ on a W*-TRO $\TX$,
then  the following equality holds:
\[ G \ltimes_{\alpha} \TX = \Fix( \Ad_{\rho} \ot \alpha). \]
\end{propn}

\begin{proof}
Assume that $\TX$ is non-degenerately represented in $B(\Hil;\Kil)$ and
consider  the extension of $\alpha$ to an action $\beta$ of $G$ on the
von Neumann algebra $\lX$ given by Lemma \ref{groupaction}. The
uniqueness of Hamana extensions shows that the action $ \Ad_{\rho} \ot
\beta$ of $G$ on $ B(L^2(G)) \wot \lX$ is the canonical extension of $
\Ad_{\rho} \ot \alpha$, the action of $G$ on $ B(L^2(G)) \wot
\TX$. The left version of Corollary 19.13 in \cite{Stratila},
attributed there to M.\,Takesaki and T.\,Digernes, shows that  $G
\ltimes_{\beta} \lX = \Fix ( \Ad_{\rho} \ot \beta)$.
In view of Proposition \ref{crossedcorner} it remains to show that
\[ \Fix( \Ad_{\rho} \ot \alpha)=  (I_{L^2(G)} \ot P_{\Kil})\, \Fix ( \Ad_{\rho} \ot \beta)\, (I_{L^2(G)} \ot P_{\Hil}).\]
This however follows from Corollary \ref{fixedcorner} in view of the
comments above.
\end{proof}

\begin{remark}
In \cite{HamanaUnpubl} Hamana defines the crossed product for an action of a group on an operator space directly via the fixed point formula of the type above. Note however that his definition does not coincide explicitly with ours, as he follows the approach of \cite{NakTak}, where everything is formulated in terms of the \emph{right} invariant Haar measure of $G$ (so that the crossed product contains the amplification of the \emph{right} group von Neumann algebra).
\end{remark}

It should be clear from the above discussions that it is also possible to develop the TRO crossed product construction in the C*-setting, we will however not need it in the sequel.

\section{Actions of locally compact quantum groups on W*-TROs and
  resulting crossed products} \label{LCQG}

In this section we discuss actions of locally compact quantum groups on W*-TROs and define associated crossed products.

We follow the von Neumann algebraic approach to locally compact
quantum groups due to Kustermans and Vaes
\cite{KV}, see also \cite{KNR} and \cite{KNR2} for more background.
A \emph{locally  compact quantum group} $\QG$, effectively a virtual object, is
studied via the von Neumann algebra $L^{\infty}(\QG)$, playing the
role of the algebra of essentially bounded measurable functions on $\QG$,
equipped with a \emph{coproduct}
$\Com:L^{\infty}(\QG) \to L^{\infty}(\QG) \wot L^{\infty}(\QG)$, which
is a unital normal coassociative $^*$-homomorphism. A locally compact
quantum group $\QG$ is by definition assumed to
admit  a left Haar weight $\phi$ and a right Haar weight $\psi$ --
these are faithful, normal semifinite weights on $L^{\infty}(\QG)$
satisfying suitable invariance conditions.
The GNS representation space for the left Haar weight will be denoted
by $\ltwo(\QG)$. All the information about $\QG$ is contained in the
\emph{right multiplicative unitary}
$V \in B(\ltwo(\QG) \ot \ltwo(\QG))$; it is
a unitary operator such that  we have
\[ \Com(x) = V(x \ot I_{L^2(\QG)}) V^*, \qquad x \in \Linf(\QG).\]
This fact enables us to define a natural extension of the coproduct,
the map $\wt{\Com}: B(\ltwo(\QG)) \to B(\ltwo(\QG) \ot \ltwo(\QG))$
given by the same formula:
\begin{equation}
\wt{\Com}(y) = V(y \ot I_{L^2(\QG)}) V^*, \qquad y \in B(\ltwo(\QG)).
\label{Comext}
\end{equation}
In fact, $\wt{\Com}$ takes values in $B(\ltwo(\QG)) \wot \Linf(\QG)$
as $V\in \Linf(\hQG)'\wot\Linf(\QG)$, where
$\hQG$ is the \emph{dual} locally compact quantum group of $\QG$
(the algebra $L^{\infty}(\hQG)$ acts naturally on $\ltwo(\QG)$).
If $\QG=G$ happens to be a locally compact group, then
$L^{\infty}(\hQG)=\VN(G)$. Finally note that by analogy with the
classical situation we denote the predual of $L^{\infty}(\QG)$ by $L^1(\QG)$.

Recall the standard definition of an action of $\QG$ on a von Neumann
algebra $\mlg$. A (continuous, left) action of $\QG$ on $\mlg$ is an
injective, normal, unital $^*$-homomorphism $\beta:\mlg \to
L^{\infty}(\QG) \wot \mlg $ satisfying the action equation
\[ (\id_{L^{\infty}(\QG)} \ot \beta)\circ \beta = (\Com_{\QG} \ot \id_{\mlg})\circ \beta.\]

Replacing a von Neumann algebra with a W*-TRO yields no extra complications.

\begin{deft}
Let $\TX$ be a W*-TRO and let $\QG$ be a locally compact quantum
group. An \emph{action} of $\QG$ on $\TX$ is an injective,  normal,
non-degenerate TRO morphism
$\alpha:\TX \to  L^{\infty}(\QG) \wot \TX$ such that
\[ (\id_{L^{\infty}(\QG)} \ot \alpha)\circ \alpha = (\Com_{\QG} \ot
\id_{\mlg})\circ \alpha.\]
\end{deft}

Theorem \ref{avatars} implies that if $\QG=G$ happens to be a
classical locally compact group, the definition above
agrees with Definition \ref{def:group-act}.

\begin{propn} \label{QHamextactions}
Let $\TX$ be a W*-TRO and let $\QG$ be a locally compact quantum group
acting on $\TX$ via $\alpha:\TX \to  L^{\infty}(\QG) \wot \TX$. Then
the extension provided by Proposition \ref{extend}
defines an action $\beta=\lX \to  L^{\infty}(\QG) \wot \lX$ of $\QG$ on $\lX$.
\end{propn}

\begin{proof}
Similar to the proof of Theorem \ref{avatars}: effectively we use the
fact that if we denote the W*-TRO $ L^{\infty}(\QG) \wot \TX$ by $\TY$,
then we have $\lY\cong L^{\infty}(\QG) \wot \lX$.
\end{proof}

If $\beta:\mlg \to  L^{\infty}(\QG) \wot \mlg$ is an action of $\QG$
on a von Neumann algebra $\mlg$, the crossed product $\QG
\ltimes_{\beta} \mlg$ is defined as the von Neumann algebra $(
(L^{\infty}(\hQG) \ot I_{\mlg})\beta (\mlg) )''$. Equivalently,
\[ \QG \ltimes_{\beta} \mlg =  \mathrm{w^*-}\textup{cl } \textup{Lin} \{(y \ot I_{\mlg})\beta(m): y \in L^{\infty}(\hQG), m \in \mlg\}.\]
The last equality amounts to the fact that the weak$^*$ closure of
$(L^{\infty}(\hQG)\ot I_{\mlg})\beta (\mlg)$ is a $^*$-subspace of
$B(L^2(\QG)) \wot \mlg$. This is a well-known fact, formulated for example in Proposition 2.3 of \cite{KS}: it can be shown
using  a simpler version of the C*-algebraic calculations in
\cite{SkalZac} after Definition 2.4.

\begin{deft} \label{DefTROcrossed}
Let $\TX$ be a W*-TRO non-degenerately represented in $B(\Hil;\Kil)$
and let $\QG$ be a locally compact quantum group acting on $\TX$ via
$\alpha:\TX \to  L^{\infty}(\QG) \wot \TX$. The \emph{W*-TRO crossed
product} $\QG \ltimes_{\alpha} \TX$  is defined as the weak$^*$ closure
of the linear span of $(L^{\infty}(\hQG) \ot I_{\Kil})\alpha(\TX)$.
\end{deft}

We will soon note that again the crossed product has several other
descriptions, but we first need to record the quantum version of
Proposition \ref{crossedcorner}.

\begin{propn} \label{Qcrossedcorner}
Let $\TX$ be a W*-TRO non-degenerately represented in $B(\Hil;\Kil)$
and let $\QG$ be a locally compact quantum group acting on $\TX$ via
$\alpha:\TX \to  L^{\infty}(\QG) \wot \TX$.  Let $\beta:\lX \to
L^{\infty}(\QG) \wot \lX$ be the action of $\QG$ on $\lX$ provided by
Proposition \ref{QHamextactions}. Then
\[
(I_{L^2(G)} \ot P_{\Kil})(\QG \ltimes_{\beta} \lX )(I_{L^2(G)} \ot
P_{\Hil}) = \QG \ltimes_{\alpha} \TX.\]
\end{propn}
\begin{proof}
Follows exactly as in the case of Proposition \ref{crossedcorner}.
\end{proof}

\begin{cor}
Under the assumptions of Definition \ref{DefTROcrossed}, we have the
following equalities:
\begin{align*}
  \QG \ltimes_{\alpha} \TX &= \mathrm{w^*-}\textup{cl }  \textup{Lin} \{(L^{\infty}(\hQG) \ot  I_{\Kil})\alpha(\TX)(L^{\infty}(\hQG)  \ot  I_{\Hil})\}
   \\&= \mathrm{w^*-}\textup{cl }  \textup{Lin} \{\alpha(\TX)(L^{\infty}(\hQG)  \ot  I_{\Hil})\}. \end{align*}
Moreover, $\QG \ltimes_{\alpha} \TX$ is the W*-TRO generated in
$B(L^2(\QG) \ot \Hil;L^2(\QG)\ot \Kil)$ by the set $(L^{\infty}(\hQG)
\ot I_{\Kil})\alpha (\TX)$. It does not depend \textup{(}up to an
isomorphism\textup{)} on the initial choice of a faithful
non-degenerate representation of $\TX$.

\end{cor}
\begin{proof}
Follows immediately from the analogous facts for the von Neumann
crossed products and Proposition \ref{Qcrossedcorner}.
\end{proof}

If $\alpha: \TX \to L^{\infty}(\QG) \wot \TX$ is an action of $\QG$ on
a W*-TRO $\TX$, then the fixed point space of $\alpha$ is defined as
\[ \Fix \alpha = \{x \in \TX: \alpha(x) = I_{L^2(\QG)} \ot x\}.\]

\begin{propn} \label{qfixedcorner}
Assume that $\TX$ is a W*-TRO non-degenerately represented in some
$B(\Hil;\Kil)$, $\alpha$ is an action of $\QG$ on $\TX$ and $\beta$ is
an action of $\QG$ on $\lX$ introduced in Proposition
\ref{QHamextactions}. Then $\Fix{\alpha} = P_{\Kil} (\Fix{\beta})
P_{\Hil}$.
\end{propn}

\begin{proof}
Follows as in Corollary \ref{fixedcorner}.
\end{proof}

Let $\alpha:\TX \to  L^{\infty}(\QG) \wot \TX$ be an action of
$\QG$ on $\TX$ and consider the following map $\delta: B(L^2(\QG))\wot \TX \to L^{\infty}(\QG) \wot B(L^2(\QG))\wot \TX$:
\begin{equation} \delta(z) = \chi_{12} (V_{12}^* (\id_{B(L^2(G))} \ot \alpha)(z)V_{12}), \qquad z \in B(L^2(\QG)) \wot \TX, \label{Qdelta}\end{equation}
where $V$ is the right multiplicative unitary of $\QG$
(compare to the formula \eqref{Adrhoalpha}, remembering that for quantum groups we denote simply by $\alpha$ what used to be $\pi_{\alpha}$).

The following result is a quantum version of Proposition \ref{crossedfixed}, this time following from the von Neumann algebraic result due to Enock \cite{Enock}, see also \cite{KNR2}.

\begin{tw}\label{EnockTRO}
Let $\TX$ be a W*-TRO  and let $\QG$ be a locally compact quantum
group acting on $\TX$ via $\alpha:\TX \to  L^{\infty}(\QG) \wot\TX$. The map
$\delta: B(L^2(\QG))\wot \TX \to L^{\infty}(\QG) \wot B(L^2(\QG))\wot\TX$
defined by \eqref{Qdelta} is an action of $\QG$
on the W*-TRO $B(L^2(\QG))\wot \TX $. Moreover we have the following
equality:
\[ \QG \ltimes_{\alpha} \TX = \Fix \delta.\]
\end{tw}
\begin{proof}
Let $\TX$ be non-degenerately represented in $B(\Hil;\Kil)$ and let
$\beta$ be the action of $\QG$ on $\lX$ provided in Theorem
\ref{QHamextactions}. Then Theorem 2.3 of \cite{KNR2}, which is a
simplified version of Theorem 11.6 of \cite{Enock}, says that the map
$\gamma: B(L^2(\QG))\wot \lX \to L^{\infty}(\QG) \wot B(L^2(\QG))\wot \lX$:
\begin{equation} \gamma(t) = \chi_{12} (V_{12}^* (\id_{B(L^2(G))} \ot \beta)(t)V_{12}), \qquad t \in B(L^2(\QG)) \wot \lX, \label{Qgamma}\end{equation}
is an action of $\QG$ on $\lX$ and
\begin{equation} \label{Fixgamma}\QG \ltimes_{\beta} \lX = \Fix \gamma.\end{equation}
It is easy to verify that in fact for
$z \in B(L^2(\QG)) \wot \TX \subset  B(L^2(\QG)) \wot \lX$ we have
\[ \delta(z) = (I_{\ltwo(\QG)} \ot I_{\ltwo(\QG)}\ot P_{\Kil})
       \gamma(z) (I_{\ltwo(\QG)} \ot I_{\ltwo(\QG)} \ot P_{\Hil}).\]
This implies, via Corollary \ref{extendconverse}, that $\delta$ is a normal TRO morphism, whose Hamana extension is $\gamma$. An explicit computation and another application of Proposition \ref{extend} show that $\delta$ is in fact an action of $\QG$ on the W*-TRO $B(L^2(\QG))\wot \TX$, with $\gamma$ clearly being the extension of $\delta$ provided by Proposition \ref{QHamextactions}. Then formula \eqref{Fixgamma} and Propositions \ref{Qcrossedcorner} and \ref{qfixedcorner} end the proof.
\end{proof}

We finish this section by discussing certain connections between the TROs arising as fixed point spaces of completely contractive maps, studied in Section 1,  and (quantum) group actions.

\begin{propn} \label{almostaction}
Let $\mlg$ be a von Neumann algebra, and let $\beta:\mlg \to \Linf(\QG)\wot
\mlg  $ be an action of a locally compact quantum group on $\mlg$.
Let $\wt{\TX}$ be a weak$^*$-closed subspace of $\mlg$, and let
$P:\mlg \to \mlg$ be a completely contractive idempotent map such
 that $P(\mlg) = \wt{\TX}$  and
\begin{equation} \beta \circ P = (\id_{\Linf(\QG)} \ot_F P) \circ \beta.  \label{commutPalpha} \end{equation}
Then the formula \textup{(}see Proposition \ref{TROconst}\/\textup{)}
\[   \alpha = (\id_{\Linf(\QG)}\wot\iota )\circ \beta \circ  \iota^{-1}\]
defines a normal, injective TRO morphism $\alpha:\TX \to \Linf(\QG) \wot \TX$ satisfying the action equation
\[(\id_{L^{\infty}(\QG)} \ot \alpha)\circ \alpha = (\Com_{\QG} \ot \id_{\mlg})\circ \alpha.\]
\end{propn}
\begin{proof}
For $x,y,z\in \wt \TX$, we have
\begin{align*}
\alpha(\{\iota(x),\iota(y),\iota(z)\}_\TX)
&=(\id_{\Linf(\QG)}\wot\iota )\circ \beta (P(xy^* z))
\\&= (\id_{\Linf(\QG)}\wot\iota )\circ (\id_{\Linf(\QG)} \ot_F P)
(\beta (x)\beta(y)^*\beta(z))\\&= \{\alpha(\iota(x)),\alpha(\iota(y)),\alpha(\iota(z))\}_ {\Linf(\QG)\wot  \TX}.
\end{align*}
Hence $\alpha$ is a TRO morphism, and it is normal
because $\iota$ is a weak*-homeomorphism.
It is also easy to check that $\alpha$ is injective and satisfies the
coassociativity condition, using the corresponding  properties of
$\beta$.
\end{proof}

\begin{remark}
We do not know whether the map constructed above is an action of $\QG$
on $\TX$, as it is not clear whether it is non-degenerate. Let us
sketch a natural approach to proving non-degeneracy, so that it is
clear where it breaks down. Using the notations of the last
proposition we should show that the weak*-closed linear span of
elements of the form
\[
 \{ \alpha(\iota(x)), a\ot \iota(y), b\ot \iota(z) \}_{\Linf(\QG)\wot  \TX},
\qquad x,y,z\in \wt \TX, a,b\in \Linf(\QG),
\]
is equal to $\Linf(\QG)\wot  \TX$.
Writing $x = P(m)$ for $m\in \mlg$, we have
\begin{align*}
\{ \alpha(\iota(x)), a\ot \iota(y), b\ot \iota(z) \}_{\Linf(\QG)\wot  \TX}
&= \{ (\id_{\Linf(\QG)}\ot \iota P)\beta(m), a\ot \iota(y), b\ot \iota(z) \}_{\Linf(\QG)\wot
  \TX} \\
&= (\id_{\Linf(\QG)} \ot \iota P)\bigl(\beta(m)(a^* b \ot y^* z)\bigr)
\end{align*}
(where we used Youngson's identity
$P(P(m_1)P(m_2)^*P(m_3))=P(m_1P(m_2)^*P(m_3))$). In other words, it
suffices to show that
$(\id \ot P) (\beta(\mlg)(\Linf(\QG) \wot \wt{\TX}^* \wt{\TX}))$
is linearly weak$^*$-dense in $\Linf(\QG) \wot \wt{\TX}$. Now we know on one hand via
Proposition 2.9 of \cite{KS} that $\beta(\mlg)(\Linf(\QG)\ot 1)$ is linearly
weak$^*$-dense in $\Linf(\QG)\wot \mlg$  and on the other hand that $P(\mlg \wt{\TX}^* \wt{\TX})$ is weak$^*$-dense in $\wt{\TX}$ (essentially because $\TX$ is a W*-TRO). Combining these two facts brings us close to completing the proof, but $P$ is not assumed to be
normal (and cannot be for applications).
Note that in the positive case (by which we mean the case where $P$ is a completely positive projection and
$\TX$ is a von Neumann algebra) the argument goes through simply by choosing
$m = 1$.

\end{remark}

 The problem disappears in the case $\QG$ is a classical group, as then we can rather use the pointwise picture of the actions.

\begin{tw} \label{actiononFixedPoints}
Let $\mlg$ be a von Neumann algebra, and let $\beta:G \to \Aut(\mlg) $ be
an action of a locally compact group $G$ on $\mlg$.
 Let $\wt{\TX}$ be a weak$^*$-closed subspace of $\mlg$, and let
 $P:\mlg \to \mlg$ be a completely contractive idempotent map such
 that $P(\mlg) = \wt{\TX}$  and
\begin{equation} \beta_g \circ P =  P \circ \beta_g, \qquad g \in G.  \label{commutPpoint} \end{equation}
Then the formula \textup{(}see Proposition \ref{TROconst}\textup{)}
\[   \alpha_g = \iota \circ \beta_g \circ  \iota^{-1}, \qquad g \in G,\]
defines an action $\alpha$ of $G$ on the W*-TRO $\TX$.
\end{tw}

\begin{proof}
The fact that each $\alpha_g$ ($g \in G$) is a normal TRO morphism follows as in the last proposition; as we have for $g,h \in G$
\[ \alpha_g \circ \alpha_h = \iota \circ \beta_g \circ  \iota^{-1} \circ  \iota \circ \beta_h \circ  \iota^{-1} =  \iota \circ \beta_g \circ  \beta_h \circ  \iota^{-1} =
 \iota \circ \beta_{gh} \circ  \iota^{-1} = \alpha_{gh}\]
 and $\alpha_e=\id_{\TX}$, each $\alpha_g$ is in fact an automorphism, and $\alpha: G \to \Aut(\TX)$ is a homomorphism. Finally the continuity condition follows from that for $\beta$: if $(g_i)_{i \in \Ind}$ is a net of elements of $G$ converging to $g \in G$ and $x\in \TX$, then, as $\iota$ is a homeomorphism with respect to weak $^*$-topologies, we have
 \begin{align*} \mathrm{w^*-}\lim_{i \in \Ind} \alpha^x(g_i)
   &= \mathrm{w^*-}\lim_{i \in \Ind} \alpha_{g_i}(x)
    = \mathrm{w^*-}\lim_{i \in \Ind} \iota (\beta_{g_i} (\iota^{-1}(x)))
    = \iota (\mathrm{w^*-}\lim_{i \in \Ind} \beta_{g_i}
    (\iota^{-1}(x)))
    \\&= \iota ( \beta_{g} (\iota^{-1}(x))) = \alpha^x(g).\end{align*}
\end{proof}

Note that although the projection $P$ features in one of the conditions both in Proposition \ref{almostaction} and Theorem \ref{actiononFixedPoints}, the actual maps constructed there depend only on its image.

\section{Poisson boundaries associated with contractive functionals in $C_0^u(\QG)^*$}

Let $\QG$ be a locally compact quantum group,
let $C_0\sp u(\QG)$  be the universal C*-algebra
associated with $\QG$, and
let $\cop_u$ be the coproduct on $C_0\sp u(\QG)$ (see  \cite{kus:univ}).
The Banach space dual of $C_0\sp u(\QG)$ will be denoted $M\sp u(\QG)$ and
called the \emph{measure algebra of $\QG$}.
It is a Banach algebra with the product defined by
\[
\mu \star \nu := (\mu \ot \nu) \circ \Com_u, \quad\mu,\nu
\in M\sp u(\QG).
\]

Given $\mu\in M\sp u(\QG)$ the associated \emph{right convolution}
operator $R_{\mu}:L\sp\infty(\QG) \to L\sp\infty(\QG)$
is defined by the formula
\[
\la R_{\mu} (x),\omega\ra = \la \omega\conv\mu, x\ra, \quad
x\in \Linf(\QG), \omega\in\lone(\QG).
\]
This is well-defined as $\lone(\QG)$ is an ideal in $M\sp u(\QG)$.
Moreover, $ R_\mu$ is normal.

We are ready to apply the results of the earlier sections to
the construction of
extended Poisson boundaries for contractive (not-necessarily
positive) quantum measures.
We say that $\mu\in M\sp u(\QG)$ is \emph{contractive} if $\|\mu\|\le 1$.

\begin{theorem} \label{extended}
Let $\mu \in M\sp u(\QG)$ be contractive.
Then the fixed point space $\Fix R_{\mu}:=\{x\in L^{\infty}(\QG):
R_{\mu}(x) = x\}$ has a unique \textup{(}up to a weak$^*$-continuous complete
isometry\textup{)} structure of a W*-TRO, which we will denote $\TX_{\mu}$.
\end{theorem}
\begin{proof}
This is an immediate consequence of Theorem \ref{CE_TRO}, as
$R_{\mu}:L^{\infty}(\QG) \to L^{\infty}(\QG)$ is a complete
contraction.
\end{proof}

We will call the space $\Fix R_{\mu}$ with the W*-TRO structure
induced via Theorem \ref{extended} an \emph{extended Poisson boundary
  associated} to $\mu$.

\begin{lem} Let $\mu \in M\sp u(\QG)$ be contractive. The extended Poisson
  boundary $\Fix R_{\mu}$ is a unital subspace of $\Linf (\QG)$ iff
  $\mu$ is a state.
\end{lem}
\begin{proof}
  Follows from the equivalence $R_{\mu}(I_{L^\infty(\QG)})=I_{L^\infty(\QG)}$
  iff $\mu(I_{MC_0^u(\QG)})=1$ iff $\mu$ is a state
  (here $MC_0^u(\QG)$ denotes the multiplier algebra of $C_0^u(\QG)$).
\end{proof}

\begin{cor} \label{trivbound}
A locally compact quantum group $\QG$ is amenable if and only if there
exists a contractive $\mu \in M\sp u(\QG)$ such that
$\Fix R_{\mu} = \bc I_{L^\infty(\QG)}$.
\end{cor}
\begin{proof}
Follows from the last lemma and Theorem 4.2 of \cite{KNR}.
\end{proof}

Given a contractive $\mu \in M\sp u(\QG)$ we can also consider an associated
convolution operator $\Theta_{\mu}$ acting on $B(\ltwo(\QG))$, defined
in \cite{JNR} (see also \cite{KNR2}): it is a unique normal completely
bounded map such that \begin{equation} \wt{\Com} \circ \Theta_{\mu} =
  (\id_{B(L\sp2(\QG))} \ot R_{\mu}) \circ \wt{\Com}. \label{ThetaL} \end{equation}

\begin{theorem} \label{BL2TRO}
Let $\mu \in M\sp u(\QG)$ be contractive.
Then the fixed point space $\Fix \Theta_{\mu}:=\{x\in B(L^2(\QG)):
\Theta_{\mu}(x) = x\}$ has a unique \textup{(}up to a weak$^*$-continuous
complete isometry\textup{)} structure of a W*-TRO, which we will denote
$\TY_{\mu}$.
\end{theorem}
\begin{proof}
This is an immediate consequence of Theorem \ref{CE_TRO}, as
$\Theta_{\mu}:B(L^2(\QG)) \to B(L^2(\QG))$ is a normal  complete contraction.
\end{proof}

Let $\mu \in M\sp u(\QG)$ be contractive. Fix a free ultrafilter $\beta$ and use
it as in Theorem \ref{CE_TRO} to construct completely contractive
projections $P$ from $\Linf(\QG)$ onto $\Fix R_{\mu}$ and $P_{\Theta}$
from $B(\ltwo(\QG))$ onto $\Fix \Theta_{\mu}$. It is then not
difficult to see that due to \eqref{ThetaL} we have also
\begin{equation} \wt{\Com} \circ P_\Theta = (\id_{B(L\sp2(\QG))}  \ot_F P) \circ
  \wt{\Com}.\label{comPlevel} \end{equation}

\begin{propn}\label{gammaembed}
Let $\QG$ and $\mu$ be as above and let $\iota: \Fix R_\mu \to
\TX_\mu$, $\kappa: \Fix \Theta_\mu \to \TY_\mu$ denote respective
weak$^*$ homeomorphisms.  Then the formula
\[ \gamma= (\id_{B(L\sp2(\QG))} \ot \iota) \circ  \wt{\Com} \circ \kappa^{-1}\]
defines an injective normal TRO morphism
$\gamma: \TY_{\mu} \to B(L\sp2(\QG)) \wot \TX_{\mu}$.
\end{propn}
\begin{proof}
This is proved similarly as Proposition \ref{almostaction}, using the intertwining relation \eqref{comPlevel}.
\end{proof}

In the case where $G$ is a classical group and $\mu \in M(G)$ is
contractive, we can in fact identify the image of the map
$\gamma$. First of all we can show via Theorem
\ref{actiononFixedPoints} that there is a natural action of $G$ on the
W*-TRO arising from $\Fix R_{\mu}$.

\begin{lem}\label{actionconv}
Let $G$ be a locally compact  group and let $\mu \in M(G)$ be
contractive. Then there is a natural action $\alpha$ of $G$ on the
W*-TRO $\TX_{\mu}$, given essentially by the left multiplication.
\end{lem}
\begin{proof}
Consider the action $\beta$ of $G$ on $L^{\infty}(G)$ given by the formula
\[ (\beta_g (f))(h) = f(g\inv h), \qquad f\in \Linf(G), g, h \in G.\]
It is then easy to check that we have $\beta_g \circ R_{\mu} =  R_{\mu} \circ \beta_g$, and so also $\beta_g \circ P = P \circ \beta_g$, where $P$ is a completely contractive projection given by the limit (along some ultrafilter) of iterates of $R_{\mu}$. Theorem \ref{actiononFixedPoints} ends the proof.
\end{proof}

We are now ready to establish the connection between the W*-TROs $\TX_{\mu}$ and $\TY_{\mu}$.

\begin{theorem} Let $G$ be a locally compact group and
let $\mu \in M(G)$ be contractive. Let $\alpha$ be the action of $G$
on the W*-TRO $\TX_{\mu}$ introduced  in Lemma \ref{actionconv}. We
then have a natural isomorphism
\[ \TY_{\mu} \cong G\ltimes_{\alpha}\TX_{\mu},\]
given by the map $\gamma$ introduced in Proposition \ref{gammaembed}.
\end{theorem}

\begin{proof}
We begin by showing that $\gamma(\TY_{\mu})$ is contained in
$G\ltimes_{\alpha} \TX_{\mu} $. By Proposition \ref{crossedfixed}
it suffices to show that for each $x \in \wt{\TY}_{\mu}$ and $g \in
G$ we have $(\Ad_{\rho} \ot \alpha)_g(\gamma(\kappa(x)))=
\gamma(\kappa(x))$.
Recall that by the definition of the action constructed in Lemma
\ref{actionconv}  we have
$\alpha_g \circ \iota = \iota \circ \beta_g$. Recall also that if we
view $\Linf(G)$ as a subalgebra of
$B(\ltwo(G))$, then the map $\beta_g$ is equal to $(\Ad_{\lambda})_g$,
where $\lambda$ denotes again the left regular representation. This
means that
\[ ((\Ad_{\rho})_g \ot \alpha_g)\bigl( (\id \ot \iota)\circ \wt{\Com}(x) \bigr)
= (\id \ot \iota) \bigl((\rho_g \ot \lambda_g) V (x \ot 1)V^* (\rho_g^* \ot \lambda_g^*) \bigr)
\]
Recall that the right multiplicative unitary for $G$ is given  by the formula
\[(Vf)(g,h) = \delta(h)\sp{1/2}f(gh,h),\qquad f \in \ltwo(G), g,h \in G,\]
where $\delta$ is the modular function of $G$.
Then an explicit calculation shows that for any $y \in B(\ltwo(G))$ we have
\[ (\rho_g \ot \lambda_g) V (y \ot 1)V^* (\rho_g^* \ot \lambda_g^*) = V (y \ot 1)V^*.\]
Thus $(\Ad_{\rho} \ot \alpha)_g(\gamma(\kappa(x)))$ does not in fact depend on $g$ and the first part of the theorem is proved.

As $\gamma$ is a normal TRO morphism, it has a weak$^*$-closed
image. Thus to show that $\gamma(\TY_{\mu})$ contains
$\TX_{\mu}\rtimes_{\alpha} G$ it suffices (by Lemma \ref{equivcrprod} and
Definition \ref{defcrprod}) to show that for every $g\in G$ and
$x \in \Fix R_{\mu}$ we have
$(\lambda_g \ot I)(\pi_{\alpha}(\iota(x)) \in\gamma(\TY_{\mu})$ (note
that the symbol $I$ above can be interpreted
explicitly once we fix a concrete representation of $\TX_{\mu}$). This
is equivalent
to proving that
$(\id \ot \iota^{-1})\left((\lambda_g \ot I)(\pi_{\alpha}(\iota(x))\right) \in
\wt{\Com}(\Fix \Theta_{\mu})$. Consider the map
$(\id\ot\iota^{-1}):B(\ltwo(G))\wot \TX_{\mu} \to B(\ltwo(G))\wot\Fix R_{\mu}$.
Note that both the domain and range spaces are in
fact $B(\ltwo(G))$ left modules in a natural way and moreover that
$(\id \ot \iota^{-1})$ is a $B(\ltwo(G))$-module map. Thus, recalling
how the action $\alpha$ was constructed in Lemma \ref{actionconv}
we can first verify that the integrated forms of actions $\alpha$ and
$\beta$ satisfy the equality
$(\id \ot \iota)\circ \pi_{\beta} = \pi_{\alpha}\circ \iota$
(remembering that $\iota$ is a homeomorphism for weak$^*$ topologies
and using equality \eqref{point>int})  and then see that
$(\id \ot \iota^{-1})\left((\lambda_g \ot I)(\pi_{\alpha}(\iota(x))\right) =
(\lambda_g \ot I)\pi_{\beta}(x)$. Now the integrated form of the
action $\beta$ is nothing but the coproduct, so we need to show simply
that $(\lambda_g \ot I)\Com(x) \in \wt{\Com}(\Fix \Theta_{\mu})$. To this
end we consider $\lambda_g x \in B(\ltwo(G))$. As $\Theta_{\mu}$ is a
$\VN(G)$-module map which extends $R_{\mu}$ and $x \in \Fix R_\mu$,
we have $\lambda_g x \in \Fix \Theta_{\mu}$. Then as the first leg of $V$
commutes with $\VN(G)$ we have $\wt{\Com} (\lambda_g x) = (\lambda_g
\ot I)\wt{\Com}(x)= (\lambda_g \ot I)\Com(x)$. This ends the proof.

\end{proof}

\begin{remark}
The analogous result for $\mu$ being a state is shown in \cite{KNR} for $G$ replaced by any locally compact quantum group. Here the stumbling block in extending the last theorem to the quantum setting is precisely the fact that we are not able to deduce in general that the map $\alpha$ constructed in  Proposition \ref{almostaction} is non-degenerate (otherwise we could use Theorem  \ref{EnockTRO} instead of Proposition \ref{crossedfixed}).
\end{remark}


\begin{thebibliography}{Ham99}


\bibitem[BLM]{BLM} D. Blecher and C. Le Merdy,
\emph{Operator algebras and their modules an operator space approach},
Oxford University Press, Oxford, 2004.

\bibitem[ChL]{Chu-Lau} C.H. Chu and A.T.-M. Lau, \emph{Harmonic functions on groups and Fourier algebras},
Lecture Notes in Mathematics, 1782. Springer-Verlag, Berlin, 2002.

\bibitem[ER$_1$]{EfR} E. Effros and Z.-J. Ruan,  \emph{Operator
  spaces}, Oxford University Press, 2000.

\bibitem[ER$_2$]{EfR-kac} E. Effros and Z.-J. Ruan,
Operator space tensor products and Hopf convolution algebras
\emph{J. Operator Theory} \textbf{50} (2003), 131--156.

\bibitem[Eno]{Enock} M. Enock,  Measured quantum groupoids in action,
\emph{M\'em. Soc. Math. Fr.\ (N.S.)}, no.\,114 (2008).

\bibitem[EOR]{effros-ozawa-ruan}
E.~G. Effros, N.~Ozawa, and Z.-J. Ruan, On injectivity and nuclearity for
  operator spaces, \emph{Duke Math. J.} \textbf{110} (2001), 489--521.


\bibitem[Ha$_1$]{Hamana}
M.~Hamana, Triple envelopes and \v {S}ilov boundaries of operator
  spaces, \emph{Math. J. Toyama Univ.} \textbf{22} (1999), 77--93.

  \bibitem[Ha$_2$]{HamanaUnpubl}
M.~Hamana, Injective envelopes of dynamical systems,  \emph{Math. J. Toyama Univ.} \textbf{34} (2011), 23--86.


\bibitem[Izu]{Izumi} M. Izumi, Non-commutative Poisson boundaries and
  compact quantum group actions, \emph{Adv. Math.} \textbf{ 169} (2002), 1--57.

\bibitem[JNR]{JNR} M. Junge, M. Neufang and Z.-J. Ruan, A
  representation theorem for locally compact quantum groups,
  \emph{Int. J. Math.} \textbf{20} (2009), 377--400.


\bibitem[KNR$_1$]{KNR} M. Kalantar, M. Neufang and Z.-J. Ruan,
  Poisson boundaries over locally compact quantum groups,
  \emph{Int. J. Math.} \textbf{24} (2013), 1350023.

\bibitem[KNR$_2$]{KNR2} M. Kalantar, M. Neufang and Z.-J. Ruan,
  Realization of quantum group Poisson boundaries as crossed products,
  \emph{Bull. Lond. Math. Soc.} \textbf{46} (2015), 1267--1275.



\bibitem[KaS]{KS} P. Kasprzak and P. So\l tan,     Quantum groups with
  projection on von Neumann algebra level, \emph{J.\,Math.\,Anal.\,Appl.} \textbf{427} (2015), 289–-306.


\bibitem[Kus]{kus:univ} J. Kustermans, Locally compact quantum groups in the universal setting,
\emph{Internat. J. Math.} \textbf{12} (2001), 289--338.


\bibitem[KuV]{KV} J. Kustermans and S. Vaes, Locally compact quantum groups
  \emph{Ann. Sci. {\'E}cole Norm. Sup. (4)} \textbf{33}
  (2000),  837--934.

\bibitem[NaT]{NakTak} Y. Nakagami and M. Takesaki, ``Duality for Crossed Products of von Neumann Algebras,'' Lecture Notes in Mathematics, 731. Springer, Berlin-Heidelberg-New York, 1979.

\bibitem[NeR]{NealRusso}
M.~Neal and B.~Russo,
Operator space characterizations of C*-algebras and ternary rings,
\emph{Pacific J. Math.} \textbf{209} (2003), 339--364.


\bibitem[NSSS]{ours} M. Neufang, P.~Salmi, A.~Skalski and N.~Spronk,
  Contractive idempotents on locally compact quantum groups,
  \emph{Indiana Univ. Math. J.} \textbf{62} (2013), 1983--2002.

\bibitem[SS$_1$]{salmi-skalski:idem}
P.~Salmi and A.~Skalski, Idempotent states on locally compact quantum
groups, \emph{Quart. J. Math.} {\bf 63} (2012), 1009--1032.


\bibitem[SS$_2$]{SSTRO}
P.~Salmi and A.~Skalski, Inclusions of ternary rings of operators and
         conditional expectations, \emph{Math. Proc. Camb. Phil. Soc.} \textbf{155} (2013), 475--482.


\bibitem[SkZ]{SkalZac}  A.~Skalski and J.~Zacharias, Approximation properties and entropy estimates for crossed
products by actions of amenable discrete quantum groups,
\emph{J. Lond. Math. Soc. (2)} \textbf{82} (2010), 184--202.


\bibitem[Sol]{Solel} B. Solel, Isometries of Hilbert C*-modules,
  \emph{Trans. Amer. Math. Soc.}  \textbf{353} (2001), 4637--4660.

\bibitem[Str]{Stratila} S. Str\u{a}til\u{a}, \emph{Modular theory in
  operator algebras}, Abacus Press, Tunbridge Wells, 1981.

\bibitem[Tak]{takesaki:vol1}
M.~Takesaki, \emph{Theory of operator algebras. {I}}, Encyclopaedia of
  Mathematical Sciences, vol. 124, Springer-Verlag, Berlin, 2002.

\bibitem[You]{Youngson}
M.~A. Youngson, Completely contractive projections on
{$C^{\ast}$}-algebras, \emph{Quart. J. Math. Oxford Ser. (2)}
\textbf{34} (1983), 507--511.

\bibitem[Zet]{Zettl}  H. H. Zettl, Ideals in Hilbert modules and
  invariants under strong Morita equivalence of C*-algebras,
  \emph{Arch. Math.} \textbf{39} (1982), 69--77.


\end{thebibliography}
\end{document}